\documentclass[a4paper,11pt]{amsart}

\usepackage{amsfonts,amssymb}
\usepackage{textcomp}

\theoremstyle{plain}

\newtheorem{thm}{Theorem}[section]

\newtheorem{pro}[thm]{Proposition}
\newtheorem{lem}[thm]{Lemma}
\newtheorem{cor}[thm]{Corollary}

\newcommand{\Z}{\mathbb{Z}}

\newcommand{\N}{\mathbb{N}}

\DeclareMathOperator{\PAut}{PAut}
\DeclareMathOperator{\Inn}{Inn}

\DeclareMathOperator{\Aut}{Aut}

\def\l{\langle}
\def\r{\rangle}

\begin{document}

\title[A finiteness condition on centralizers]%
{A finiteness condition on centralizers\\ in locally nilpotent groups}

\author[G.A. Fern\'andez-Alcober]{Gustavo A. Fern\'andez-Alcober}
\address{Matematika Saila\\ Euskal Herriko Unibertsitatea UPV/EHU\\
48080 Bilbao, Spain. {\it E-mail address}: {\tt gustavo.fernandez@ehu.eus}}

\author[L. Legarreta]{Leire Legarreta}
\address{Matematika Saila\\ Euskal Herriko Unibertsitatea UPV/EHU\\
48080 Bilbao, Spain. {\it E-mail address}: {\tt leire.legarreta@ehu.eus}}

\author[A. Tortora]{Antonio Tortora}
\address{Dipartimento di Matematica\\ Universit\`a di Salerno\\
Via Giovanni Paolo II, 132\\ 84084 Fisciano (SA)\\ Italy. {\it E-mail address}: {\tt antortora@unisa.it}}

\author[M. Tota]{Maria Tota}
\address{Dipartimento di Matematica\\ Universit\`a di Salerno\\
Via Giovanni Paolo II, 132\\ 84084 Fisciano (SA)\\ Italy. {\it E-mail address}: {\tt mtota@unisa.it}}

\thanks{The first two authors are supported by the Spanish Government, grants
MTM2011-28229-C02-02 and MTM2014-53810-C2-2-P, and by the Basque Government, grant IT753-13. The last two authors would like to thank the Department of Mathematics at the University of the Basque Country for its excellent hospitality while part of this paper was being written; they also wish to thank G.N.S.A.G.A. (INdAM) for financial support.}

\keywords{Centralizers, locally nilpotent groups\vspace{3pt}}

\begin{abstract}
We give a detailed description of infinite locally nilpotent groups $G$ such that the index $|C_G(x):\l x\r|$ is finite, for every $\langle x \rangle \ntriangleleft G$. We are also able to extend our analysis to all non-periodic groups satisfying a variation of our condition, where the requirement of finiteness is replaced with a bound.
\end{abstract}

\maketitle

\section{Introduction}

In this paper we continue our study, introduced in \cite{FLTT3}, of groups with a finiteness condition on centralizers of elements, which was in turn inspired by the results in \cite{FLTT,FLTT2} and \cite{zha-gao,zha-guo}. Following \cite{FLTT3}, we say that a group $G$ is an FCI-group provided that
\begin{equation}\label{FCI}
|C_G(x):\langle x \rangle| < \infty
\quad
\text{for every $\langle x \rangle \ntriangleleft G$,}
\end{equation}
and if there exists $n$ such that
\begin{equation}\label{BCI}
|C_G(x):\langle x \rangle| \le n
\quad
\text{for every $\langle x \rangle \ntriangleleft G$,}
\end{equation}
then we say that $G$ is a BCI-group. It is easy to see that free groups satisfy the FCI-condition for every non-trivial element, and it is also known that the same happens with hyperbolic torsion-free groups \cite{Gr}. Apart from the abelian case, such groups are examples of FCI-groups which are not BCI-groups (Lemma \ref{aperiodic}). Of course Dedekind groups vacuously satisfy the BCI-condition. Further examples of BCI-groups are the generalized dihedral groups \cite[Example 2.1]{FLTT3}, and the Tarski monster groups (i.e. infinite simple $p$-groups, for $p$ a prime, all of whose proper non-trivial subgroups are of order $p$). The classes of FCI- and BCI-groups are closed under taking subgroups but, since not all groups are FCI-groups, they are not usually closed under taking quotients. However, the quotient of an FCI-group by a finite normal subgroup is again an FCI-group, and similarly for a BCI-group \cite[Proposition 2.2]{FLTT3}.

\vspace{8pt}

For a periodic group $G$, conditions (\ref{FCI}) and (\ref{BCI}) imply that $|C_G(x)|<\infty$ for every $\langle x \rangle \ntriangleleft G$. Recall that there are well-known results giving information about the structure of an infinite group simply from the knowledge that the centralizer of one element is finite
(see \cite[page 263]{Sh}, and \cite{Shu} for an account on the topic). In our context, by imposing conditions on a broad set of centralizers, the results are much more precise. In \cite{FLTT3} we dealt with infinite locally finite FCI-groups, and with periodic BCI-groups. More precisely, we showed that every infinite locally finite FCI-group is a BCI-group, and it is either a Dedekind group, or an extension of a certain infinite periodic Dedekind group $D$ by an appropriate power automorphism of $D$ of finite order. We also proved that locally graded periodic BCI-groups are locally finite, where the restriction to locally graded groups is imposed to avoid the presence of Tarski monster groups.

\vspace{8pt}

Our goal in the present paper is to extend our earlier work to infinite locally nilpotent FCI-groups, and to all non-periodic BCI-groups. As in the realm of locally finite groups, we are able to give a detailed description of such groups, which are again certain cyclic extensions of Dedekind groups. Actually, in the case of periodic locally nilpotent groups (Theorem \ref{periodic locally nilpotent}), we can get an easier description than the characterization of infinite locally finite FCI-groups \cite[Theorem 4.4]{FLTT3}. Furthermore, as it turns out, there is a huge difference between non-periodic locally nilpotent FCI-groups and BCI-groups (compare Theorem \ref{non-periodic locally nilpotent} with Corollary \ref{non-periodic locally nilpotent BCI}). Observe also that, since free groups are FCI-groups, one cannot expect detailed results about the structure of non-periodic FCI-groups in general.

\vspace{8pt}

\noindent
\textit{Notation\/}.
We use mostly standard notation in group theory. In particular, if $G$ is a finite $p$-group and $k\geq 1$, we denote by $\Omega_k(G)$ the subgroup generated by all elements of $G$ of order at most $p^k$. If $G$ is a periodic group, we write $\pi(G)$ for the set of prime divisors of the orders of the elements of $G$; if $G$ is also nilpotent and $\varphi$ is an automorphism of $G$, we denote by $G_p$ the unique Sylow $p$-subgroup of $G$, and by $\varphi_p$ the restriction of $\varphi$ to $G_p$. Finally, $R^{\times}$ stands for the group of units of a ring $R$, and $\Z_p$ for the ring of $p$-adic integers.

\section{Power automorphisms}

In this section we collect some preliminary results on power automorphisms of a group that will be needed in the sequel. Recall that an automorphism of a group $G$ is said to be a \emph{power automorphism\/} if it sends every element $x\in G$ to a power of $x$. Power automorphisms form an abelian subgroup of $\Aut G$, which we denote by $\PAut G$.

The following lemma can be easily checked.

\begin{lem}
\label{PAut Q8}
An automorphism of the quaternion group is a power automorphism if and only if it is inner.
\end{lem}

The next result is due to Robinson (see \cite[Lemma 4.1.2]{Ro3}).

\begin{thm}
\label{PAut A robinson}
Let $A$ be an abelian group. Then the following hold:
\begin{enumerate}
\item
If $A$ is not periodic then $\PAut A$ is of order $2$, generated by the inversion automorphism.
\item
If $A$ is periodic then $\PAut A$ is isomorphic with the cartesian product of the groups $\PAut A_p$, where $p$ runs over $\pi(A)$. Furthermore, if $A$ is a $p$-group for a prime $p$, then $\PAut A \cong R^{\times}$, where $R=\Z_p$ if $\exp A=\infty$ and $R=\Z/p^n\Z$ if $\exp A=p^n<\infty$.
In both cases, an isomorphism can be obtained by sending every $t\in R^{\times}$ to the automorphism given by $a\mapsto a^t$ for all $a\in A$.
\end{enumerate}
\end{thm}

Let $A$ be an abelian $p$-group, and let $\varphi\in\PAut A$. According to (ii) of the previous theorem, every power automorphism of $A$ can be written in the form $a\longmapsto a^t$ for some $t\in\Z_p^{\times}$.
Simply observe that, if $\exp A=p^n<\infty$, it suffices to consider a representative in $\Z$ of the element in $\Z/p^n\Z$ given by Theorem \ref{PAut A robinson}. The $p$-adic integer $t$ is uniquely determined if $\exp A=\infty$, and it is only determined modulo $p^n$ if $\exp A=p^n<\infty$.
In the sequel, we write $\exp \varphi=t$, if $\varphi(a)=a^t$ for every $a\in A$.

\begin{lem}
\label{valua}
Let $\varphi$ be a power automorphism of an abelian $p$-group $A$ of finite rank. If $t=\exp \varphi$ is of infinite order in $\Z_p^{\times}$, then
$C_A(\varphi^k)$ is finite for every $k\geq 1$.
\end{lem}

\begin{proof}
Given $k\geq 1$, let us write $t^k=1+\sum_{i\ge j} \, t_ip^i$ where $j\geq 0$, $t_i\in \{0,\ldots, p-1\}$ for every $i$
and $t_j\ne 0$.
Then $a\in C_A(\varphi^k)$ if and only if $a^{p^j}=1$,
since $\sum_{i\ge j} \, t_ip^{i-j}$ is an invertible element in $\Z_p^{\times}$.
Hence $C_A(\varphi^k)=\Omega_{j}(A)$. Since $A$ is of finite rank, it follows that $C_A(\varphi^k)$ is finite.
\end{proof}

Recall that the rank of an abelian $p$-group is the dimension, as a vector space over
the field with $p$ elements, of the subgroup formed by the elements of order at most $p$. In that case, having finite rank is equivalent to requiring that the group is a direct sum of finitely many cyclic and quasicyclic groups \cite[4.3.13]{Ro}.
More generally, if $A$ is an abelian group, the $p$-rank of $A$ is defined as the rank of
$A_p$, and the torsion-free rank, or $0$-rank, of $A$ is the cardinality of a maximal independent subset of elements of $A$ of infinite order.
These are denoted by $r_p(A)$ and $r_0(A)$, respectively.
The total rank of $A$ is
\[
r(A) = r_0(A) + \sum_p \, r_p(A)
\]
where $p$ runs over all prime numbers, and by definition a soluble group $G$ is said to have finite abelian total rank if it has an abelian series whose factors are of finite total rank. Finally, a group $G$ has finite (Pr\"ufer) rank $r$ if every finitely generated subgroup of $G$ can be generated by $r$ elements and $r$ is the least such integer. For more details we refer to \cite[page 84]{LR}.

\section{Locally nilpotent FCI-groups}

This section is devoted to infinite locally nilpotent FCI-groups.
We split the determination of such groups into two cases, according as the group is periodic or not.

If an infinite locally nilpotent FCI-group is periodic, then Theorem 4.4 of \cite{FLTT3} is available. Nevertheless, in this case, it will be easier to apply that theorem to an appropriate Sylow subgroup than to the whole FCI-group. We supply for the reader's convenience the characterization of an infinite locally finite $p$-group which is also an FCI-group.

\begin{lem}
\label{determination locally finite FCI-groups}
Let $G$ be a non-Dedekind $p$-group, for $p$ a prime. Then $G$ is an infinite locally finite FCI-group if and only if $p=2$ and $G=\langle g,A \rangle$, where $A$ is infinite abelian of finite rank, and $g$ is an element of order at most $4$ such that $g^2\in A$ and $a^g=a^{-1}$ for all $a\in A$.
\end{lem}

\begin{proof}
By Theorem 4.4 of \cite{FLTT3}, $G$ is an infinite locally finite FCI-group if and only if $p=2$ and $G=\langle g,A \rangle$, where $A$ is infinite abelian of finite rank, and the action of $g$ on $A$ is given by a power automorphism of order $m>1$ such that $|G:A|=m$. In particular $A$ is of infinite exponent, so that $m=2$ by (ii) of Theorem \ref{PAut A robinson}, that is, $g$ acts on $A$ as the inversion. Then $C_A(g)=\Omega_1(A)$ and, since $g^2\in C_A(g)$, we conclude that the order of $g$ is at most~$4$.
\end{proof}

\begin{thm}
\label{periodic locally nilpotent}
Let $G$ be a non-Dedekind infinite periodic group.
Then $G$ is a locally nilpotent FCI-group if and only if $G=P\times Q$, where $P$ and $Q$ are as follows:
\begin{enumerate}
\item
$P=\langle g, A\rangle$ is a $2$-group, where $A$ is infinite abelian of finite rank, and $g$ is an element of order at most $4$ such that $g^2\in A$ and $a^g=a^{-1}$ for all $a\in A$.
\item
$Q$ is a finite abelian $2'$-group.
\end{enumerate}
\end{thm}

\begin{proof}
Assume first that $G$ is a locally nilpotent FCI-group.
Since $G$ is not a Dedekind group, there exists a Sylow $p$-subgroup $P$ with a non-normal cyclic subgroup $\l x\r$. If we write $G=P\times Q$ with $Q$ a $p'$-group, then $Q$ is finite, since it is contained in $C_G(x)$ and $G$ is a periodic FCI-group. Thus $P$ is an infinite locally finite $p$-group which is also an FCI-group, and (i) follows from Lemma \ref{determination locally finite FCI-groups}. Since all elements of $Q$ generate a normal subgroup of $G$, then $Q$ is a Dedekind $2'$-group, and in particular it is abelian.

Conversely, let $G=P\times Q$ as in the statement. It is obvious that $G$ is locally nilpotent. Furthermore, one can readily see that $G$ is a group of the type described in (ii) of Theorem 4.4 of \cite{FLTT3}, and so $G$ is an FCI-group.
\end{proof}

Next we consider the non-periodic case. The following lemma is well-known and an easy exercise.

\begin{lem}
\label{normal is central}
Let $G$ be a locally nilpotent group, and let $x\in G$ be an element of infinite order.
If $\langle x \rangle \lhd G$ then $x\in Z(G)$.
\end{lem}

If $G$ is a nilpotent group and the centre of $G$ has finite exponent, then $G$ itself has finite exponent (see, for example,
\cite[5.2.22 (i)]{Ro}). The same proof applies to show that, if $N$ is a normal subgroup of $G$ such that
$N\cap Z(G)$ is of finite exponent, then so is $N$. This result will be needed in the proof of the following lemma.
We also recall that, according to a theorem of Mal'cev,
a soluble group of automorphisms of a finitely generated abelian group is polycyclic (see \cite[Theorem 3.27]{Ro2}).

\begin{lem}
\label{fg}
Let $G$ be a nilpotent FCI-group. Then either $G$ is a Dedekind group, or $G$ is finitely generated.
\end{lem}

\begin{proof}
Suppose that there exists $x\in G$ such that $\l x\r \ntriangleleft G$, and that $G$ is not finitely generated. Let $A$ be a maximal abelian normal subgroup of $G$. Then $A=C_G(A)$ \cite[5.2.3]{Ro} and $G/A$ embeds in $\Aut A$. It follows that $A$ cannot be finitely generated, by the above result of Mal'cev, and therefore $\l a\r\lhd G$ for every $a\in A$: in fact, if there exists $a\in A$ such that $\l a\r\ntriangleleft G$, then $|A:\l a\r|\leq |C_G(a):\l a\r|$ is finite and $A$ is finitely generated. Let $T$ be the torsion subgroup of $A$, and suppose $T<A$. Applying Lemma \ref{normal is central}, we get $A\leq Z(G)$. Thus $|A:A\cap \l x\r|$ is finite and $A$ is finitely generated, a contradiction. Hence $A$ must be periodic, in particular $A\cap \l x\r$ is finite. Now $G$ is nilpotent, and so every element of $A$ of prime order lies in $Z(G)$. Furthermore we have
\[
|A\cap Z(G)|
\le
|C_A(x)|
=
|C_A(x):A\cap \langle x \rangle| |A\cap \langle x \rangle|<\infty.
\]
Thus $A$ has finitely many elements of prime order and, consequently, it is of finite total rank.
Finally, as mentioned before the lemma, since $G$ is nilpotent and $A\cap Z(G)$ is finite, $A$ is of finite exponent.
This allows us to conclude that $A$ is finite, a contradiction.
\end{proof}

\begin{lem}
\label{tf}
Let $G$ be a locally nilpotent FCI-group, and let $T$ be the torsion subgroup of $G$. Suppose that $G$ is non-abelian, and $T<G$. Then:
\begin{enumerate}
\item All elements of $T$ generate a normal subgroup of $G$, and $T$ is a Dedekind group given by the direct product of finitely many $p$-groups of finite rank.
\item The quotient group $G/T$ is cyclic.
\end{enumerate}
\end{lem}

\begin{proof}
Let $x\in T$ and $y\in G\setminus T$. Since $G$ is not abelian, we may assume that $y\not\in Z(G)$. Then, by Lemma \ref{normal is central}, we have $\langle y \rangle \ntriangleleft G$.

(i)
Put $H=\langle x,y \rangle$, and suppose that $\langle x \rangle \ntriangleleft G$.  Then $C_G(x)$ is finite and, since the torsion subgroup of $H$ is finite, $x$ has finitely many conjugates in $H$, i.e.\ $|H:C_H(x)|$ is finite. It follows that $H$ is finite, which is a contradiction. This proves that $\langle x \rangle \lhd G$. Now, if $x$ is of prime order, then $x\in Z(H)$. Thus all elements of $T$ of prime order commute with $y$. Since $|C_T(y)|\leq |C_G(y):\l y\r|$ is finite, then $\pi(T)$ is finite and $T_p$ is of finite rank, for every $p\in \pi(T)$.

(ii) First we prove that $G/T$ is abelian. Clearly, we may assume that $G$ is finitely generated. Then $G$ is nilpotent, say of class $c$, and $T$ is finite. Thus $G/T$ is an FCI-group \cite[Proposition 2.2]{FLTT3} and therefore we may also assume that $G$ is torsion-free. Suppose for a contradiction that $G$ is not abelian, and consider a maximal abelian normal subgroup $A$ of $G$. If $a\in A$ and $\langle a \rangle \ntriangleleft G$, then $|A:\langle a \rangle|$ is finite. Since $A$ is torsion-free, we obtain that $A$ is cyclic. But $A\lhd G$, so that $\langle a \rangle\lhd G$, which is a contradiction.
Thus every cyclic subgroup of $A$ is normal in $G$. Therefore $G/A$ embeds in $\PAut A$ and, by Theorem \ref{PAut A robinson}, $|G:A|=2$ and $a^g=a^{-1}$, for all $a\in A$ and $g\in G\smallsetminus A$. It follows that $1=[a,_c g]=a^{(-2)^c}$, which is impossible.

Next we show that $G/T$ is locally cyclic. Let us consider a finitely generated subgroup $H/T=\langle h_1T,\ldots, h_r T\rangle$ of $G/T$, and let us prove that $H/T$ is cyclic. To this end, we put $K=\langle h_1,\ldots,h_r,y \rangle$ and prove that $KT/T$ is cyclic. Since $K'$ is finitely generated, nilpotent, and contained in $T$, it follows that $K'$ is finite.
Hence $K$ is a finitely generated FC-group, i.e. the conjugacy class of every element of $K$ is finite, and then $|K:Z(K)|$ is finite
(see Exercise 14.5.2 of \cite{Ro}).
On the other hand,
\[
|Z(K)\langle y \rangle:\langle y \rangle|
\le
|C_G(y):\langle y \rangle|
\]
is also finite, since $\langle y \rangle \ntriangleleft G$.
Thus $|K:\langle y \rangle|$ is finite, and consequently $\langle yT \rangle$ has finite index in $KT/T$.
Since $KT/T$ is a torsion-free abelian group, we deduce that $KT/T$ is cyclic, as desired.

Now, taking also into account (i), we obtain that $G$ has an abelian series in which the $p$-rank of each factor ($p=0$ or a prime) is finite. In other words, $G$ is a soluble group of finite abelian total rank.
By applying 10.4.3 of \cite{LR}, there exists a nilpotent subgroup $N$ of $G$ such that
$NT$ has finite index in $G$. It follows that $N\cap T$ is a proper subgroup of $N$, and so $N$ can be generated by aperiodic elements. We claim that $N$ is finitely generated. Otherwise $N$ must be a Dedekind group, by Lemma \ref{fg}, and since it is not periodic, $N$ is abelian.
But $G$ is an FCI-group and $N$ is not finitely generated, hence all elements of $N$ generate a normal subgroup of $G$, and by Lemma \ref{normal is central}, we have $N\le Z(G)$. Then $|N: N\cap \langle y \rangle|$ is finite, and $N$ is finitely generated, a contradiction. This completes the proof of the claim.
Since $G/NT$ is finite, it follows that $G/T$ is finitely generated.
But we know that $G/T$ is locally cyclic, and so we conclude that $G/T$ is cyclic.
\end{proof}

It is easy to see that every periodic group containing an infinite abelian subgroup, and in particular
every infinite locally finite group \cite[14.3.7]{Ro}, is abelian whenever the index $|C_G(x):\langle x \rangle|$ is finite for every non-central
element $x\in G$. This can be extended to infinite locally nilpotent groups as follows.

\begin{cor}
\label{central ln}
Let $G$ be an infinite locally nilpotent group.
If $|C_G(x):\langle x \rangle|$ is finite for every $x\in G\smallsetminus Z(G)$, then $G$ is abelian.
\end{cor}

\begin{proof}
Suppose for a contradiction that $G$ is not abelian. Let $T$ be the torsion subgroup of $G$. We may assume that $T<G$. Then $T$ is non-trivial, by (ii) of Lemma \ref{tf}. Let $x\in T$. By (i) of Lemma \ref{tf} we have $\l x\r \lhd G$, and so $|G:C_G(x)|$ is finite. Applying the hypothesis, we get $x\in Z(G)$. Hence $T\le Z(G)$ and, by (ii) of Lemma \ref{tf}, $G/Z(G)$ is cyclic. Thus $G$ is abelian, contrary to our assumption.
\end{proof}

We are now ready to characterize the non-periodic locally nilpotent FCI-groups.

\begin{thm}
\label{non-periodic locally nilpotent}
Let $D$ be a periodic Dedekind group such that $\pi(D)$ is finite and $D_p$ is of finite rank for every $p\in\pi(D)$.
Let $\varphi$ be a power automorphism of $D$, and write $t_p=\exp \varphi_p$ whenever $D_p$ is abelian.
Assume that the following conditions hold:
\begin{enumerate}
\item
If $D_2$ is non-abelian, $\varphi_2$ is the identity automorphism.
\item
If $p>2$ then $t_p\equiv 1 \pmod p$, and if $D_p$ is infinite also $t_p\neq 1$.
\item
If $p=2$ and $D_2$ is infinite, then $t_2\neq 1, -1$.
\end{enumerate}
Then the semidirect product $G=\langle g\rangle\ltimes D$, where $g$ is of infinite order and acts on $D$ via $\varphi$, is a locally nilpotent FCI-group.
Conversely, every non-periodic locally nilpotent FCI-group is either abelian or isomorphic to a group as above, $D$ being the torsion subgroup of $G$.
\end{thm}

\begin{proof}
First of all, let $G=\langle g \rangle \ltimes D$ be as in the statement of the theorem.
We begin by proving that $G$ is locally nilpotent.
By (i), we can extend the definition of $t_2$ also to the case when $D_2$ is not abelian, by taking $t_2=1$.
Observe that the condition $t_p\equiv 1\pmod p$, which is required in (ii) for all odd $p$, also holds for $p=2$, since any $2$-adic unit is congruent to $1$ modulo $2$.
As a consequence, if $d\in D_p$ is of order $p^n$, we have
\[
[d,g,\overset{n}{\ldots},g]
=
d^{\,(t_p-1)^n}
=
1.
\]
Assume that $p$ is odd, or that $p=2$ and $D_2$ is abelian. Since $\l d\r \lhd G$, we deduce that $d$ belongs to the $n$th term, $Z_n(G)$, of the upper central series of $G$. On the other hand, by (i), if $D_2$ is not abelian then $d\in Z_2(G)$. Thus, for every $d\in D_p$, there exists $n=n(d)$ such that $d\in Z_{n}(G)$. Let $Z=\bigcup_{n\geq 1}Z_n(G)$ be the hypercenter of $G$. Of course, $Z$ is locally nilpotent. Furthermore $D\leq Z$, so that $G/Z$ is cyclic and we conclude that $G$ is locally nilpotent.

Now we prove that $G$ is an FCI-group.
Since $\varphi\in\PAut D$ and $D$ is a Dedekind group, every element of $D$ generates a normal subgroup of $G$.
Hence we only need to show that $|C_G(x):\langle x \rangle|$ is finite for every
$x\in G\smallsetminus D$.
Let us write $x=g^kd$, with $k\in\Z\smallsetminus\{0\}$ and $d\in D$.
Notice that
\[
\vert C_G(x):\langle x\rangle\vert
=
\vert C_G(x)D:\langle x\rangle D\vert \, \vert C_G(x)\cap D:\langle x\rangle \cap D\vert
\leq\vert G:\langle x\rangle D\vert \vert C_D(x)\vert
\]
where the first factor of the right-hand side is finite, since $G/D$ is infinite cyclic and $x\not\in D$.
On the other hand, we have
\[
C_D(x) = \prod_{p\in\pi(D)} \, C_{D_p}(x),
\]
and since $\pi(D)$ is finite, in order to prove that $C_D(x)$ is finite, it suffices to see that $C_{D_p}(x)$ is finite whenever $D_p$ is infinite.
So let us fix a prime $p$ such that $D_p$ is infinite.
If $p=2$ then this implies that $D_2$ is abelian, since $D_2$ is of finite rank and $D$ is a Dedekind group.
Thus we always have $D_p\le Z(D)$ for $p$ as above.
As a consequence, $C_{D_p}(x)=C_{D_p}(g^k)=C_{D_p}(\varphi_p^k)$.
We may assume without loss of generality that $k\ge 1$.
Now, by conditions (ii) and (iii) of the statement of the theorem, $t_p$ is an element of infinite order in $\Z_p^{\times}$.
Then we conclude from Lemma \ref{valua} that $C_{D_p}(\varphi_p^k)$ is finite.
This completes the proof that $G$ is an FCI-group.

Conversely, let now $G$ be a non-periodic locally nilpotent FCI-group, and assume that $G$ is not abelian.
Let $D$ be the torsion subgroup of $G$. Then, by Lemma \ref{tf}, all elements of $D$ generate a normal subgroup of $G$. Also, $D$ is a Dedekind group given by the direct product of finitely many $p$-groups of finite rank, and $G/D$ is infinite cyclic. Thus we have $G=\langle g \rangle \ltimes D$, for some  $g\in G$.
Let $\varphi$ be the power automorphism of $D$ induced by conjugation by $g$, and let
$t_p=\exp \varphi_p$ whenever $D_p$ is abelian, for every $p\in\pi(D)$. Since $G$ is generated by the coset $gD$, there exists $d\in D$ such that $gd\notin Z(G)$. Then
$\langle gd \rangle \ntriangleleft G$, by Lemma \ref{normal is central}, and so $C_{D_p}(gd)$ is finite.
Furthermore, since $G$ is locally nilpotent, we have $\Omega_1(D_p)\leq C_{D_p}(gd)$ for all $p\in \pi(D)$. Let us suppose either that $p>2$ or that $p=2$ and $D_2$ is infinite (and so abelian). Then $D_p\le Z(D)$ and so $C_{D_p}(g)=C_{D_p}(gd)$ is finite and contains $\Omega_1(D_p)$.
As a consequence, $t_p$ cannot be $1$ if $D_p$ is infinite.
On the other hand, if $x\in D_p$ is of order $p$ then $\varphi(x)=x^{t_p}$ is equal to $x$, and so
$t_p\equiv 1\pmod p$. This proves (ii), and part of (iii).

Next we show that (i) holds. Assume that $D_2$ is not abelian and write $D_2=Q\times A$, with $Q$ isomorphic to the quaternion group of order 8 and $A$ elementary abelian of finite rank.
We claim that $G=D_2C_G(Q)$. Observe that
\begin{equation}
\label{inequalities}
|D_2:C_{D_2}(Q)| = |D_2:C_G(Q)\cap D_2| \le |G:C_G(Q)| \le |\PAut Q|.
\end{equation}
Now $C_{D_2}(Q)=Z(Q)\times A$ has index $4$ in $D_2$, and by Lemma \ref{PAut Q8},
$\PAut Q=\Inn Q$ is of order $4$.
Hence the first inequality in (\ref{inequalities}) is an equality, and as a consequence we get
$G=D_2C_G(Q)$. Thus we can choose the element $g$ of the previous paragraph in $C_G(Q)$.
Since $g$ acts as a power automorphism on the elementary abelian $2$-group $A$, we conclude that
$\varphi_2$ is the identity automorphism, as desired.

Finally, we complete the proof of (iii) by showing that $t_2\ne -1$ if $D_2$ is infinite.
Note that $D_2$ is abelian and $\exp D_2=\infty$. By way of contradiction, assume that $t_2=-1$.
Let $d\in D_2$ be an element of order greater than $2$. Then $[g^2d,g]=[d,g]=d^{-2}\ne 1$ and consequently $\langle g^2d\rangle\ntriangleleft G$.
Hence $C_{D_2}(g^2d)$ is finite. But, on the other hand, since $D_2$ is abelian and $g^2\in C_G(D_2)$, we have  $C_{D_2}(g^2d)=D_2$, which is infinite.
This contradiction completes the proof of the theorem.
\end{proof}

The following corollary is analogous to Corollary 4.2 of \cite{FLTT3} for infinite locally finite FCI-groups, and the proof is exactly the same, so we omit it.

\begin{cor}
Let $G$ be an infinite locally nilpotent FCI-group.
Then $G$ is metabelian.
\end{cor}

\section{BCI-groups}
\label{sec:BCI}

In this section we restrict our attention to BCI-groups. First we recall that, in the periodic case, every locally graded BCI-group is locally finite \cite[Theorem 5.3]{FLTT3}. This, together with Theorem \ref{periodic locally nilpotent}, provides the following result.

\begin{pro}
Let $G$ be an infinite locally graded $p$-group, where $p$ is an odd prime.
If $G$ is a BCI-group then $G$ is abelian.
\end{pro}

Now we describe non-periodic BCI-groups. First we need the following lemma.

\begin{lem}
\label{aperiodic}
Let $G$ be a BCI-group and let $x\in G$.
If $x$ is of infinite order, then $\langle x\rangle \lhd G$. In particular, if $G$ is torsion-free, then it is abelian.
\end{lem}

\begin{proof}
Suppose that $\l x\r \ntriangleleft G$. Let $n$ be a positive integer such that
$|C_G(g):\langle g\rangle|\leq n$ for every $\l g\r \ntriangleleft G$, and let $p>n$ be any prime number. Since $|\l x\r :\l x^p\r|=p>n$, then $\langle x^{p} \rangle \lhd G$. Similarly, if $q>n$ is another prime different from $p$, we have $\langle x^{q} \rangle \lhd G$. It follows that $\langle x \rangle=\l x^p,x^q\r\lhd G$, a contradiction.
\end{proof}

\begin{thm}
\label{non-periodic BCI}
Let $G$ be a non-periodic group.
Then the following conditions are equivalent:
\begin{enumerate}
\item
$G$ is a BCI-group.
\item
There exists $n\in\N$ such that $|C_G(x)|\le n$ whenever $\langle x \rangle\ntriangleleft G$.
\item
Either $G$ is abelian or $G=\langle g, A \rangle$, where $A$ is a non-periodic abelian group of finite $2$-rank and $g$ is an element of order at most $4$ such that $g^2\in A$ and $a^g=a^{-1}$ for all
$a\in A$.
\end{enumerate}
\end{thm}

\begin{proof}
Obviously, (i) follows from (ii).
Next we prove that (i) implies (iii).
Assume that $G$ is a non-periodic BCI-group, and that $G$ is not abelian.
Let $X$ be the set of all elements of $G$ which generate a normal subgroup of $G$, and put
$A=\langle X \rangle$.
By Lemma \ref{aperiodic}, $X$ contains all elements of $G$ of infinite order.
If $x\in X$ then $G'\le C_G(x)$, and so we get $G'\le C_G(A)$.
Consequently $A'\le Z(A)$ and $A$ is nilpotent. Then applying (i) of Lemma \ref{tf} to $A$, it follows that every element of $A$ generates a normal subgroup of $A$. Since $A$ is non-periodic, $A$ must be abelian. As a consequence, $A$ is a proper subgroup of $G$.

Let us now consider an element $x$ of infinite order.
Then $x\in A$ and $A\le C_G(x)$.
If there exists $y\in C_G(x)\smallsetminus A$ then $x\in C_G(y)$.
But $\langle y \rangle \ntriangleleft G$ and $y$ is of finite order, so that $C_G(y)$ is finite, which is a contradiction.
Hence $C_G(x)=A$, and therefore $G/A$ is cyclic of order $2$.
Choose an element $g\in G\smallsetminus A$.
Then $G=\langle g, A \rangle$ and $x^g=x^{-1}$.
Since $x$ is an arbitrary element of infinite order, and these elements generate the abelian group $A$, we have $a^g=a^{-1}$ for all $a\in A$.
Thus $C_A(g)$ coincides with the set of all elements of $A$ of order at most $2$. It follows that $g^4=1$ and, since $C_A(g)$ is finite, $A$ is of finite $2$-rank.

We conclude by showing that (iii) implies (ii).
Let $G=\l g,A\r$ be of the form described in (iii), and assume that $\langle x \rangle \ntriangleleft G$.
Then $x\not\in A$, and we can write $x=ga$ for some $a\in A$.
We have $C_A(x)=C_A(g)=\{ a\in A\mid a^2=1\}$, which is finite of order $2^r$, say, since $A$ is of finite $2$-rank.
Consequently $|C_G(x)|\le |G:A||C_A(x)|\le 2^{r+1}$, which proves (ii).
\end{proof}

We finish with two more results. The first is just an application of the previous theorem.

\begin{cor}
\label{non-periodic locally nilpotent BCI}
Let $G$ be a non-periodic locally nilpotent group. If $G$ is a BCI-group then $G$ is abelian.
\end{cor}

Recall that in the realm of locally finite groups, FCI- and BCI-conditions are equivalent \cite[Corollary 4.5]{FLTT3}. In contrast, Theorem \ref{non-periodic locally nilpotent} and Corollary \ref{non-periodic locally nilpotent BCI} show that there exist non-periodic locally nilpotent FCI-groups which are not BCI-groups.

Notice also that, contrary to the case of infinite locally finite and locally nilpotent groups (see, for example, Corollary \ref{central ln}), where knowing that $|C_G(x):\langle x \rangle|$ is finite for all $x\in G\smallsetminus Z(G)$ was enough to conclude that a group $G$ is abelian, we cannot say so if $G$ is only non-periodic. Indeed, the infinite dihedral group is an example of such a group which is not abelian. Nevertheless, this is not the case if the requirement of finiteness is replaced with a bound.

\begin{pro}
Let $G$ be a non-periodic group, and assume that for some $n$ we have
$|C_G(x):\langle x \rangle|\le n$ for every $x\in G\smallsetminus Z(G)$.
Then $G$ is abelian.
\end{pro}

\begin{proof}
Let $x\in G$ be an element of infinite order. Arguing as in the proof of Lemma \ref{aperiodic}, we obtain that $x$ is central.
On the other hand, if $g\in G$ is of finite order, then $g=x^{-1}(xg)$ is also central.
Thus $G$ is abelian.
\end{proof}


\newpage

\begin{thebibliography}{10}

\bibitem{FLTT}
G.\,A.\ Fern\'andez-Alcober, L.\ Legarreta, A.\ Tortora, and M.\ Tota,
A restriction on centralizers in finite groups,
\textit{J.\ Algebra\/} \textbf{400} (2014), 33--42.

\bibitem{FLTT2}
G.\,A.\ Fern\'andez-Alcober, L.\ Legarreta, A.\ Tortora, and M.\ Tota,
Some restrictions on normalizers or centralizers in finite $p$-groups,
{\tt arXiv:1311.2608 [math.GR]},
to appear in \textit{Israel J.\ Math.\/}

\bibitem{FLTT3}
G.\,A.\ Fern\'andez-Alcober, L.\ Legarreta, A.\ Tortora, and M.\ Tota,
A finiteness condition on centralizers in locally finite groups,
{\tt arXiv:1510.03733 [math.GR]}.

\bibitem{Gr}
M.\ Gromov, Hyperbolic groups, {\it Essays in group theory}, 75--263, Math.\ Sci.\ Res.\ Inst.\ Publ.\ {\bf 8}, Springer, New York, 1987.

\bibitem{LR}
J.\,C.\ Lennox and D.\,J.\,S.\ Robinson,
{\it The theory of infinite soluble groups\/},
Oxford University Press, Oxford, 2004.

\bibitem{Ro3}
D.\,J.\,S.\ Robinson,
Groups in which normality is a transitive relation,
\textit{Proc.\ Camb.\ Phil.\ Soc.\/} {\bf 60} (1964), 21--38.

\bibitem{Ro2}
D.\,J.\,S.\ Robinson, {\it Finiteness conditions and generalized
soluble groups}, Part 1, Springer-Verlag, Berlin, 1972.

\bibitem{Ro}
D.\,J.\,S.\ Robinson,
{\it A course in the theory of groups\/},
second edition, Springer-Verlag, New York, 1996.

\bibitem{Shu} P.\ Shumyatsky,
Centralizers in locally finite groups, {\it Turkish J. Math.} {\bf 31} (2007), 149--170.

\bibitem{Sh}
V.\,P.\ Shunkov,
On periodic groups with an almost regular involution,
\textit{Algebra and Logic\/} \textbf{11} (1972), 260--272.

\bibitem{zha-gao}
Q.\ Zhang, and J.\ Gao,
Normalizers of nonnormal subgroups of finite $p$-groups,
\textit{J. Korean Math. Soc.\/} \textbf{49} (2012), 201--221.

\bibitem{zha-guo}
X.\ Zhang, and X.\ Guo,
Finite $p$-groups whose non-normal cyclic subgroups have small index in their normalizers,
\textit{J. Group Theory\/} \textbf{15} (2012), 641--659.
\end{thebibliography}
\end{document}